\documentclass[11pt,a4paper]{article}

\usepackage[left=3cm, top=3cm,bottom=3cm,right=3cm]{geometry}

\usepackage{mathtools,amssymb,amsthm,mathrsfs,calc,graphicx,stmaryrd,xcolor,dsfont,bbm,float,array,epsfig,hyperref,color}
\usepackage[numbers,square]{natbib}
\usepackage[british]{babel}
\usepackage{amsfonts,amsmath,amsthm,amssymb}
\usepackage{verbatim}
\usepackage{multirow}
\usepackage{tikz}

\numberwithin{equation}{section}
\numberwithin{figure}{section}

\newtheorem{thm}{Theorem}[section]
\newtheorem{lem}[thm]{Lemma}

\newtheorem{prop}[thm]{Proposition}

\theoremstyle{remark}
\newtheorem{rem}[thm]{Remark}

\theoremstyle{definition}

\numberwithin{equation}{section}

\def\P{{\mathbb{P}}}

\def\S{{\mathcal{S}}}

\def\R{{\mathbb{R}}}

\newcommand{\dod}{\overset{\rm{d}}{\longrightarrow}}

\newcommand{\PP}{\mathbb{P}}
\newcommand{\E}{\mathbb{E}}
\newcommand{\ES}{\mathbb{E}_{\mathcal S}}

\renewcommand{\H}{\mathcal{H}}

\renewcommand{\P}{\mathbb{P}}
\newcommand{\NN}{\mathbb{N}}
\newcommand{\N}{\mathbb{N}}

\newcommand{\eps}{\varepsilon}
\newcommand{\s}{\sigma}
\newcommand{\8}{\infty}
\newcommand{\wt}{\widetilde}

\newcommand{\Ind}{\mathbf{1}_}

\newcommand{\un}{{\bf 1}_{U_n}}
\newcommand{\unn}{{\bf 1}_{U_{n+1}}}

\begin{document}

\title{\bfseries Limit theorems for  supercritical branching process in~random environment }
	\author{Dariusz Buraczewski, Ewa Damek}

\maketitle
\begin{abstract}

We consider  the branching process in random environment $\{Z_n\}_{n\geq 0}$, which is a~population growth process
where individuals reproduce independently of each other with the reproduction law randomly picked at each generation.
We focus on the supercritical case, when the process survives with a positive probability and grows exponentially fast on the
nonextinction set. Our main is goal is establish Fourier techniques for this model, which allow to obtain a number of precise estimates
related to limit theorems. As a consequence we provide new results concerning central limit theorem, Edgeworth expansions and  renewal theorem for $\log Z_n$.

\end{abstract}

\section{Introduction}\label{sec:intro}

A branching process in random environment (BPRE) is a population growth process
where individuals reproduce independently of each other with the 	reproduction law randomly picked at each generation.
The process was introduced by Smith and Wilkinson~\cite{smith1969branching} as a natural generalization of the classical Galton-Watson process.
To define BPRE formally, let $Q$ be a random measure on the set of non-negative integers $\NN_0 = \{0,1,\ldots\}$, that is a measurable function taking values in  the set
$\mathcal{M} = \mathcal{M}(\NN_0)$ of all probability measures on $\NN_0$ equipped with the total variation distance. Then a sequence of independent identically distributed (iid) copies of $Q$, say
$\mathcal{Q}=\{ Q_n \}_{n \geq 0}$  is called a random environment. The sequence $Z=\{Z_n\}_{n \geq 0}$  is
a~branching process in random environment $\mathcal{Q}$ if  $Z_0 =1$,
\begin{equation}\label{eq:s9}
	Z_{n+1} {=} \sum_{k=1}^{Z_n} \xi_{k,n},
\end{equation}
and given $\mathcal{Q}$, for every $n$, random variables $\{\xi_{k,n}\}_{ k \geq 1}$ are iid and independent of $Z_n$ with common distribution $Q_n$.
 For a more detailed discussion regarding BPRE itself, we recommend  the classical book of Athreya and Ney~\cite{athreya2012branching} or the recent monograph of Kersting and
Vatutin~\cite{kersting:vatutin:book}.\medskip

Asymptotic behavior of the process  $Z$ is usually determined by the environment.
Let
$A_k = \sum_{j=0}^\8 j Q_k(j)$
be the mean of the reproduction law and let $\Pi_n$ denote  the quenched expectation of $Z_n$, i.e.~$\Pi_n =	\E[Z_n \: | \: \mathcal{Q}] = \prod_{k=0}^{n-1}A_k$.
The process $Z$ survives with positive probability only in the supercritical case  $0<\E \log A <\infty$\footnote{Except of the trivial case $\P[Z_1=1]=1$.}, i.e.~when  the associated random walk
$S_n = \log \Pi_n$
drifts a.s. to $+\infty$ (Proposition 2.1 in \cite{kersting:vatutin:book}).  Then   the population grows exponentially fast on the survival set ${\mathcal S}$ \cite{Tanny} and so, it is convenient  to consider the sequence $\{\log Z_n\}_{n\geq 0}$ rather than $\{Z_n\}_{n\geq 0}$. It turns out that the behavior of $\log Z_n$ and $S_n$ is comparable and both processes admit similar limit properties. In particular, for the sequence $\log Z_n$ classical limit theorems are valid:
\begin{itemize}
\item law of large numbers:  $\frac{\log Z_n}{n}\to \mu:=\E\log A$ on $\S$ in probability, see \cite{Tanny77};
\item central limit theorem: if $Z_1 \ge 1$ a.s. and $\sigma^2 := \E\log^2A - \mu^2<\infty$, then  \newline $\frac{\log Z_n - n\mu}{\sqrt n \sigma}  \dod N(0,1)$,  see \cite{huang:liu1},
\item precise large deviations: $\P[\log Z_n > \rho n] \sim {c e^{-n\Lambda^*(\rho)}}/{\sqrt n}$, for $\rho >\mu$ and some rate function
$\Lambda^*$ (of course, some further hypotheses are required), see \cite{buraczewski:dyszewski}.
\end{itemize}
A number of further results including Berry-Esseen estimates and Cram\'er's type large deviation expansion have been recently proved in \cite{Grama:Liu}.

The purpose of the present paper is to study common features of two processes $\{\log Z_n\}_{n\geq 0}$ and $\{S_n\}_{n\geq 0}$ in more details and to obtain the Edgeworth expansion and the renewal theorem for $\{ \log Z_n\} _{n\geq 0}$.
Moreover, $\P (Z_1=0)>0$ is allowed, which, in particular, gives the central limit theorem with weaker assumptions. Our approach is based on the classical Fourier analysis.
We prove
that on the level of characteristic functions both processes are comparable, i.e.~their Fourier transforms admit the same asymptotic expansions
near 0 (Proposition \ref{prop:phi}). This observation allows essentially to
make use of Fourier techniques for the sequence $\{\log Z_n\}_{n\geq 0}$ and to obtain results far more refined than before.

\section{Main Results}\label{sec:results}

\subsection{Expansion of $\E [Z_n^{is}]$}

Let $P_0$ be a law on  the set $\mathcal M$ of all probability measures on $\N_0$.  Then the probability measure  $P = P_0^{\otimes \N_0}$ on ${\mathcal M}^{\otimes \N_0}$ defines the environment ${\mathcal Q}$.  Given the environment ${\mathcal Q}$, let $(\Gamma, \mathcal{G}) = (\NN_0^{\NN_0}, \mathcal{B}or(\NN_0^{\NN_0}))$ be the canonical probability space under which the process $Z$ is defined and let $\PP_{\mathcal{Q}}$ be the corresponding probability measure.
 Then, the total probability space is $(\Gamma \times {\mathcal M}^{\otimes \N_0}, \P)$, for $\P = \P_{\mathcal Q}\rtimes P$
 defined by $\P(A\times B) = \int_B \P_{\mathcal Q}(A)P(d{\mathcal Q})$, for $A\in \Gamma$, $B\in {\mathcal M}^{\otimes \N_0}$.
We will occasionally write $\PP[ \cdot \: | \: \mathcal{Q}] = \PP_{\mathcal{Q}}[\cdot]$.

 Our standing assumption is
\begin{equation}\label{mainass}
0<  \mu=\E \log A <\infty
\end{equation}
 i.e. BPRE $Z$ is supercritical and therefore it survives with positive probability. Moreover we assume two other conditions
\begin{itemize}
\item[(H1)] there are  $q > 1$, $p\in (1,2]$ such that
$$\E\bigg[\big( 1+|\log A_0|^q\big)\bigg(\bigg(\frac{Z_1}{A_0}\bigg)^p + 1\bigg)\bigg] <\infty;$$
  \item[(H2)] 
  $\P_{\mathcal Q}[Z_1=0] < \gamma $ a.s. for some $\gamma < 1$.
  \end{itemize}
In contrast to previous papers on limit theorems
for the supercritical BPRE \cite{Grama:Liu,huang:liu1,huang:liu2}
we allow $\P[Z_1=0]>0$ and so with positive probability the subcritical environment occurs. Probability of extinction of the process $\{Z_n\}_{n\geq 0}$ is positive but strictly smaller than $1$ a.s. due to (H2), see \cite{athreya:karlin:1971}.
 A careful examination of large deviation results and proofs in \cite{BansBoninglower} and \cite{BansBoningsmall} allows us to go beyond the restriction $Z_1\ge 1$ a.s. with as weak assumptions as (H1) and (H2).

Observe that hypothesis (H2) entails that $A$ is bounded from below, that is
\begin{equation}\label{eq:a}
   A \ge 1-\gamma,\qquad P \mbox{ a.s.}
\end{equation}
and hence all negative moments of $A$ are finite, i.e. $\E A^{-\varepsilon}<\infty$ for any $\varepsilon > 0$. Moreover, if $\eps $ is sufficiently close to $0$,
$$\E A^{-\varepsilon}<1.$$
Indeed, in view of \eqref{mainass}, the derivative of the function $\beta \mapsto \E A^{\beta} $ at 0 is strictly positive.

Let $\lambda$ be the characteristic function of $\log A$, i.e.
$$
\lambda(s) = \E[e^{is\log A}] = \E[A^{is}].
$$
Define
$$
\phi_n(s) =  \frac{\E_\S[e^{is\log Z_n}]}{\E[e^{is\log \Pi_n}]}
= \frac{\ES[Z_n^{is}]}{\lambda(s)^n},
$$ where $\E_\S[\cdot] = \E[\cdot|\S]$ denotes the expected value conditioned on the survival set $\S$.

\medskip

The following Proposition  plays a crucial role in the proof of our limit theorems and it may be viewed as the main novelty of the paper:
\begin{prop}\label{prop:phi}
Assume that hypotheses (H1) and (H2) are satisfied and let $K=\lfloor q-1\rfloor$. Then there are  $\eta>0$,  a function $\phi\in C^{K}(I_{\eta})$ defined on $I_{\eta} = (-\eta,\eta)$ and constants $\rho<1$, $C<\infty$ such that
\begin{equation}\label{eq:o3}
  \big| \phi_n^{(j)}(s)  - \phi^{(j)}(s)  \big| \le C\rho^n,\qquad s\in I_{\eta},\ j=0,\ldots,K.
\end{equation}
\end{prop}
Formula \eqref{eq:o3} indicates that the characteristic function of $\log Z_n$ is comparable with the characteristic function of $S_n$, the sum of i.i.d.
random variables $\{\log A_i\}_{i\geq 0}$. In consequence, both processes share the same limit behaviour. Similar techniques  have been widely used
to study limit theorems related to  complex random combinatorial objects (see e.g. \cite{FMN,KMS} for description of mod-$\phi$ convergence and related techniques) or general Markov chains (see e.g. \cite{herve} for the  spectral Nagaev-Guivarc'h method).
The proof of Proposition \ref{prop:phi},  contained in section 3, is inspired by the  methods introduced in \cite{BansBoninglower,damek:natert:kolesko,Grama:Liu}.

\subsection{Limit theorems}

 Now we state our limit theorems concerning the sequence $\{\log Z_n\}_{n \ge 0}$.  We start with the central limit theorem:

\begin{thm}[Central Limit Theorem]\label{thm:clt}
  Assume that $\E |\log A|^2 <\infty$. Assume moreover that conditions (H1) and (H2) are satisfied. Then
  $$
  \frac{\log Z_n - n\mu}{\sqrt n \sigma} \overset{{\rm d}}{\to} N(0,1)
  $$ conditionally on the survival set ${\mathcal S}$, where
  $\mu=\E \log A$, $\sigma^2 = {\rm Var}(\log A) = -\lambda''(0)+\lambda'(0)^2$.
\end{thm}

Under hypothesis $Z_1 \ge 1$  the central limit theorem was proved in \cite{huang:liu1}. Moreover in the succeeding paper \cite{Grama:Liu} some further results, including Berry-Esseen estimates, were established.  However, the methods used there cannot be applied when $Z_1$ may vanish. In \cite{huang:liu1} and \cite{Grama:Liu} the arguments depend strongly on negative moments of the martingale limit $W=\lim _{n\to \8} Z_n/\Pi_n$. In our situation $W$ may be zero with positive probability, thereupon its negative moments cannot be defined.

Our method is more direct and allows to mimic proofs for i.i.d.~random variables based on Fourier techniques. To present the potential of  the approach we provide two other results concerning the asymptotic behavior of $\log Z_n$: the Edgeworth expansion and the renewal theorem. These more accurate results require controlling characteristic
functions not only near 0, as in Proposition \ref{prop:phi}, but also outside a small neighbourhood of the origin. Therefore we need to assume
below that $\log A$ is nonlattice (then $\lambda(s)\not= 1$ for all $s\in \R \setminus \{0\}$). Also, to prove the higher order Edgeworth expansion,
we assume below condition \eqref{eq:pp1} saying that $\lambda(s)$ cannot approach 1 arbitrary close for $s$ tending to infinity. Both these additional hypotheses are required in the
classical situation while studying sums of i.i.d. random variables.

The first result describes the asymptotic expansion in  the central limit theorem in terms of
the Edgeworth series. Note that, since the function $\lambda$ is continuous and $\lambda(0)=1$,  the logarithm of $\lambda$ is well defined  in some neighbourhood of $0$.  Therefore, we may write $\Lambda(s) = \log \lambda(s) = \log \E A^{is}$ for $s\in I_{\eta}$, for some small $\eta>0$.
\begin{thm}[Edgeworth expansion]\label{thm:edg}
Assume $\log A$ is nonlattice and  conditions (H1) with $q\ge 4$ and (H2) are satisfied. Let $r$ be any positive integer number such that
$r\in [3,q-1]$.
If $r\ge 4$ assume additionally that
\begin{equation}\label{eq:pp1}
\limsup_{|s|\to \infty} |\E A^{is}| <1.
\end{equation}
Then
\begin{equation}
\label{eq:ss1}
\P\bigg( \frac{\log Z_n - n\mu}{\sigma \sqrt n} \le x \Big| {\mathcal S} \bigg)
=  G_r(x) + o(n^{-r/2+1}),
\end{equation}
where
\begin{equation}\label{eq:pp2}
G_r(x)= \Psi(x) - \psi(x) \sum_{k=3}^r n^{-k/2+1}Q_k(x),
\end{equation}
$\psi(x) = \frac1{\sqrt{2\pi}} e^{-x^2/2}$ is the density of a standard Gaussian variable, $\Psi(x)=\int_{-\infty}^x \psi(y)dy$ is the corresponding cumulative distribution function, $Q_k$ is a polynomial of order $k-1$ depending on the first $k$ moments of $Z_1$ and $A_1$ (but  independent of $n$ and
$r$) and  $o(n^{-r/2+1})$ denotes a function of order smaller than $n^{-r/2+1}$ uniformly with respect to $x$.
 In particular, for $r=3$, we have
	\begin{equation}
\label{eq:ss1}
\P\bigg( \frac{\log Z_n - n\mu}{\sigma \sqrt n} \le x \Big| {\mathcal S} \bigg)
= \Psi(x) - \frac{\psi(x)}{\sqrt n} \bigg( \frac{i\Lambda'''(0)}{6\sigma^3}(1-x^2) + \frac{i\phi'(0)}{\sigma}    \bigg)+ o(n^{-1/2}).
\end{equation}
\end{thm}
 \begin{rem}
$Q_k$ may be written explicitly, see \eqref{qpoly} below.
\end{rem}

Our next result is a version of the renewal theorem for $\log Z_n$:
\begin{thm}[Renewal theorem for $\log Z_n$] \label{thm:ren}
  Assume that $\log A$ is nonlattice and hypotheses (H1) and (H2) are satisfied. Then, for all reals $0\le B<C<\infty$
  $$
  \lim_{y\to\infty}\ES\big[  \# \{ n: \log Z_n \in y+[B,C]\}  \big]
=   \lim_{y\to\infty}\ES\big[  \# \{ n: e^By\le Z_n \le e^C y\}  \big] = \frac 1{\mu}(C-B).
  $$
\end{thm}

\section{Proof of Proposition \ref{prop:phi}}
Before we proving Proposition \ref{prop:phi}, we present a few auxiliary results. Some of them are quite involved. Even in the supercritical case, the process $Z$  may extinct and
the case $\P[Z_1=0]>0$ requires always a more careful analysis. In particular, then the proofs of Lemmas \ref{lem:lower} and \ref{zet} become considerably more complex.
So, they are postponed to the Appendix to make the whole argument more transparent.

\medskip

Nevertheless, the typical realization of the process grows exponentially fast on the survival set.  Lemmas \ref{lem:lower}, \ref{lem:lower ld1} and \ref{zet} below provide some qualitative analysis  of this observation.
We formulate first in our settings a version of a large deviation lemma.
 The conclusion of Lemma \ref{lem:lower} is much weaker than the large deviation result of \cite{BansBoninglower} but
 is sufficient for our purposes and allows us to work under weaker assumptions.

\begin{lem}\label{lem:lower}
Assume that   $\P[Z_1=0]>0$ and condition (H2) holds. 
Then there are $0 < \theta < \E \log A$,  $\beta >0$ and $C<\infty$ such that
\begin{equation}\label{eq:q10}
  \P\big[ 1\le Z_n \le e^{\theta n} \big] \le C  e^{-\beta n}.
\end{equation}
\end{lem}

Denote by $U_n$ the event that BPRE $Z$ survives up to the $n$th generation:
$$U_n = \{Z_n>0\}.$$

\begin{lem}\label{lem:lower ld1}
Assume hypothesis (H2) is satisfied. 
  Then  for  any $\delta>0$ there are $C,\beta>0$ such that
  \begin{equation}\label{eq: unn}
    \E \big[ Z_n^{-\delta} \un
     \big]
    \le C e^{-\beta n}.
  \end{equation}
\end{lem}
\begin{proof}
{\sc Case 1.} Consider first the case when  $Z_1\ge 1$ a.s.
Then $\E Z_1^{-\delta}<1$ and we apply an argument borrowed from \cite{Grama:Liu} (see the proof of Lemma 2.4).
It is sufficient to justify
\begin{equation}\label{eq:ksr1}
  \E Z_n^{-\delta} \le \big(\E Z_1^{-\delta}\big)^n.
\end{equation}
We proceed by induction. Since the function $f(y)=y^{-\delta}$ is convex on $(0,\8)$, by the Jensen
inequality and \eqref{eq:s9}, we have
		\begin{align*}
		\E Z_{n+1}^{-\delta}=&\E \bigg[Z_n^{-\delta}\bigg (\sum _{i=1}^{Z_n}\frac{\xi_{i,n}}{Z_n}\bigg)^{-\delta}\bigg]
			\leq  \E \bigg[Z_n^{-\delta}\sum _{i=1}^{Z_n}\frac{\xi_{i,n}^{-\delta}}{Z_n}\bigg]\\
			=&\E \bigg[Z_n^{-\delta -1 }\E \bigg[\sum _{i=1}^{Z_n}\xi_{i,n}^{-\delta}\big |Z_n \bigg]\bigg]
			= \E Z_n^{-\delta } \cdot \E \xi_{i,n}^{-\delta}
			=\E Z_n^{-\delta }\cdot \E Z_1^{-\delta}.
		\end{align*}
This proves \eqref{eq:ksr1}. 
\medskip

{\sc Case 2.} Assume now that $\P[Z_1 = 0]>0$. Then
by Lemma \ref{lem:lower},
\begin{align*}
  \E \big[ Z_n^{-\delta} \un  \big] &  \le  \P [ 1\leq Z_n\leq e^{\theta n}] + \E\big[  Z_n^{-\delta } \Ind {\{ Z_n> e^{\theta n}\}} \big]\\
  & \le C  e^{-\beta n} + e^{- \delta\theta n },
\end{align*}
which completes the proof.
\end{proof}

The next result measures deviation of the process $Z_n$ from its conditional mean. Let us define
\begin{equation}\label{eq:delta}
  \Delta_n  = \frac{Z_{n+1}}{A_n Z_n}.
\end{equation}
Formula \eqref{eq:delta} makes sense only when $Z_n$ is nonzero and so we will always refer to this random variable on the set $U_{n}$. Since $\E[Z_{n+1}|{\mathcal Q}, Z_n] = A_n Z_n$ we expect some kind of concentration property and thus $\Delta_n$ should be close to 1.

\begin{lem}\label{lem:sr1}
If conditions (H1) and (H2) are satisfied, then
 there are constants $C>0$ and $\rho<1$ such that
$$
\E\big[ (\log Z_{n+1})^j |\log(A_nZ_n)|^k|\log \Delta_n|  \unn \big] \le C \rho^n
$$ for any $n, j,k \in\N_0$ and $j+k+1< q$.
\end{lem}
\begin{proof}
We start with proving that
 there are constants $C>0$ and $\rho<1$ such that for any $r \in [0,q]$
\begin{equation}\label{eq:srrr}
  \E\big[   |\log(A_n Z_n)|^r |\Delta_n-1|^p {\bf 1}_{U_n}
  \big] \le C \rho^n,
\end{equation}
where the parameters $q,p$ are defined in condition (H1).

In view of \eqref{eq:s9} and \eqref{eq:delta}, on the set $U_{n}$, we may write
$$
\Delta_n = 	\frac{Z_{n+1}}{A_n Z_n}= \frac{1}{Z_n}\sum _{i=1}^{Z_{n}}\frac{\xi_{i,n}}{A_n}.
$$
Now we apply
the Marcinkiewicz-Zygmund inequality (see e.g. Theorem 5.1 in \cite{Gut})
to estimate the conditional expectation of $|\Delta_n-1|$ on the set $U_n = \{Z_n>0\}$:
\begin{align*}
 \E\big[ |\log A_n|^r |\Delta_n-1|^p\big|  Z_n\big]&=
\E \bigg[\frac{ |\log A_n|^r }{Z_n^p}\bigg| \sum _{i=1}^{Z_{n}}\bigg (\frac{\xi_{i,n}}{A_n}-1\bigg)
	\bigg|^p  \Big| Z_n\bigg]\\
	&\leq C  Z_n^{1-p} \E \bigg [ |\log A_n|^r \bigg |\frac{\xi_{1,n}}{A_n}-1
	\bigg |^p \bigg ]
\end{align*}
Therefore, invoking independence of $Z_n$ and $(A_n,\xi_{1,n})$, since $U_n = {\bf 1}_{\{Z_n > 0\}}$ we obtain
\begin{equation*}
\begin{split}
\E\big[   |\log(A_nZ_n)|^{r} & | \Delta_n-1|^p  \un  \big]
\le
      C \E\bigg[ \Big(  \big|\log A_n\big|^{r} + \big(\log Z_n\big)^{r}\Big)
      Z_n^{1-p}\un \bigg |\frac{\xi_{1,n}}{A_n}-1 \bigg |^p \bigg ]\\
      &\le C \E\Big[ \big( 1+  \big(\log Z_n\big)^{r}\big) Z_n^{1-p}  \un \Big]\cdot
       \E\Big[ \big(  1+ \big|\log A_0\big|^{r} \big)  \bigg |\frac{Z_1}{A_0}-1 \bigg |^p \bigg ],
\end{split}
\end{equation*} thus combining condition (H1) with  Lemma \ref{lem:lower ld1}  we obtain  \eqref{eq:srrr}.

To prove the Lemma, note that since $Z_{n+1} = \Delta_n A_n Z_n$, on the set $U_{n+1}$ we have
  $$
  \log Z_{n+1} \le |\log \Delta_n|  {\bf 1}_{\{\Delta_n \ge 1/2\}} + |\log (A_n Z_n)|.
  $$ Moreover on the set $\{\Delta_n\ge 1/2\}\cap U_n$ we have $Z_{n+1}>0$, hence $\unn = \un$.
  Therefore,  it is sufficient to prove
  \begin{equation}\label{eq:wt1}
    \E\Big[   |\log(A_nZ_n)|^{k} |\log \Delta_n|^{j+1} {\bf 1}_{\{\Delta_n \ge 1/2\}} \un  \Big] < C\rho^n
  \end{equation} for $k < q $
  and
  \begin{equation}\label{eq:wt2}
    \E\Big[   |\log(A_nZ_n)|^{k} |\log \Delta_n| {\bf 1}_{\{\Delta_n < 1/2\}} \unn  \Big] < C\rho^n
  \end{equation} for $k+1\le q$.
  We start with inequality \eqref{eq:wt1}. Assume first that  $j\ge 1$ and take
 $p \in (1,2)$ as in condition (H1). Then there is a constant $C<\infty$ such that
  \begin{equation}\label{eq:g6}
|\log(1+x)|^{j+1} \le C |x|^p
\end{equation} for any  $x\ge -1/2$. The above inequality gives
  \begin{equation}\label{eq:sr3}
  \E\big[   |\log(A_nZ_n)|^{k} |\log \Delta_n|^{j+1} {\bf 1}_{\{\Delta_n \ge 1/2\}} \un  \big]
  \le C
  \E\big[   |\log(A_nZ_n)|^{k} | \Delta_n-1|^{p}  \un  \big],
  \end{equation}
thus \eqref{eq:srrr} entails \eqref{eq:wt1} if $j\ge 1$.

If $j=0$, then choose $1<r<p$ such that  $rk\leq q$. Using the Jensen inequality we obtain
$$
    \E\big[   |\log(A_nZ_n)|^{k} |\log \Delta_n| {\bf 1}_{\{\Delta_n \ge 1/2\}} \un  \big]
\le \Big(    \E\big[   |\log(A_nZ_n)|^{rk} |\log \Delta_n|^r {\bf 1}_{\{\Delta_n \ge 1/2\}} \un  \big]\Big)^{\frac 1r}.
$$ Now, one can proceed as in the previous case, that is, first apply inequality \eqref{eq:g6}
(with $p$ replaced by $r$)
and then again \eqref{eq:srrr}. Thus we obtain \eqref{eq:wt1}.

To estimate \eqref{eq:wt2} note that   on the set
$\{\Delta_n <1/2\}\cap U_{n+1}$
we have
$$
\frac 12 > \Delta_n = \frac{Z_{n+1}}{A_n Z_n} \ge \frac{1}{A_n Z_n}.
$$
Therefore,
 $|\log(A_nZ_n)|\ge |\log \Delta_n|$,
       $2^p|\Delta_n-1|^p\ge 1$ and so  
\begin{align*}
       \E\Big[   |\log(A_nZ_n)|^{k} |\log \Delta_n| {\bf 1}_{\{\Delta_n < 1/2\}} \unn  \Big]
 & \le 2^p \E\big[   |\log(A_nZ_n)|^{k+1} |\Delta_n-1|^p \un  \big].
 \end{align*}
Referring again to  \eqref{eq:srrr} we complete the proof of \eqref{eq:wt2} and the Lemma follows.
\end{proof}

\begin{lem}\label{zet}
	Suppose that hypotheses (H1) and (H2) are satisfied. Then for every positive integer $k<q$ there is $C=C(k)$ such that
	\begin{equation}\label{logmoment}
	\E \big[ (\log Z_n)^k \Ind {U_n}\big] \leq Cn^{k}.
	\end{equation}
\end{lem}
The proof of Lemma \ref{zet} will be postponed to the Appendix.

\begin{lem}\label{lem:lower ld2}
Assume hypotheses (H1) and (H2) are satisfied.
  Then there are constants $\beta, C>0$ such that
  	\begin{equation}\label{eq:un111}
	\P (U_n\setminus \mathcal{S})\leq Ce^{-\beta n}
	\end{equation}
and   for every integer $k\in [1,q)$
  \begin{equation}\label{eq: un11}
    \E\big[ (\log Z_n)^{k}  {\bf 1}_ {U_n\setminus \S }\big] \le C e^{-\beta n}
  \end{equation}
  and
  \begin{equation}\label{eq: un1}
    \E\big[ |\log A_n |^k  {\bf 1}_ {U_n\setminus \S }\big] \le C e^{-\beta n}.
  \end{equation}
\end{lem}
\begin{proof}
If $Z_1\ge 1$ a.s.,  the process survives with probability 1 and
 $U_n=\S$, so the Lemma  holds trivially.

 Assume now that condition $\P[Z_1 = 0]>0$ holds. Since $\S = \bigcap _{n=1}^{\8}U_n$,
 we have
 $$
 U_n\setminus \S = \bigcup_{k=n}^{\infty} (U_k\setminus U_{k+1}).
 $$
Thus to prove \eqref{eq:un111}, it is enough to show that
	\begin{equation*}
	\P (U_n\setminus U_{n+1})\leq Ce^{-\beta n}.
	\end{equation*}
Indeed, by Lemma \ref{lem:lower} and (H2), we have
\begin{align*}
   \P [U_n\setminus U_{n+1}] &\le \P\big[ 1\le Z_n \le e^{\theta n} \big] + \P\big[ Z_n > e^{\theta n} \mbox{ and } Z_{n+1}=0  \big]\\
   &\le C e^{-\beta n} + \E\big[ \P[\xi_{n,1} = 0]^{Z_n} {\bf 1}_{\{ Z_n > e^{\theta n} \}}  \big]\\
   &\le C e^{-\beta n} + \gamma^{e^{\theta n}} \le 2C e^{-\beta n}
\end{align*}
and the proof of \eqref{eq:un111} is complete.

Inequalities \eqref{eq: un11} and \eqref{eq: un1} follow now from \eqref{eq:un111}, Lemma \ref{zet} and the H\"older inequality with parameters $s,t>1$ such that $1/s+1/t=1$ and
$ks <  q$. Indeed,
$$
  \E[ |\log A_n|^{k}  {\bf 1}_ {U_n\setminus \S}]
  \le  \E[ |\log A_n|^{{k}s}]^{1/s} \cdot  \P[ U_n\setminus \S  ]^{1/t} \le C e^{-\beta n/t},
$$
$$
\E[ (\log Z_n)^{k}  {\bf 1}_ {U_n\setminus \S}]
\le  \E[ (\log Z_n)^{{k}s} {\bf 1}_{U_n}]^{1/s} \cdot  \P[ U_n\setminus \S  ]^{1/t} \le C n^{(ks+1)/s}e^{-\beta n/ t}.
$$ Adjusting the value of $\beta$, i.e. replacing $\beta/t$ by $\beta$, we conclude \eqref{eq: un11} and \eqref{eq: un1}.
\end{proof}


\begin{proof}[Proof of Proposition \ref{prop:phi}]
{\sc Step 1.} 
 We are going to prove  that
  \begin{equation}\label{eq:s1}
    \big| \phi_{n+1}(s) -\phi_n(s) \big| < C \rho_0^n,\quad  s\in I_{\eta_0}
  \end{equation} for some $\rho_0<1$, $C>0$ and $\eta_0>0$.  This entails existence of $\phi $. Indeed,
  	$\phi_n$ is a Cauchy sequence uniformly with respect to  $s\in I_{\eta_0}$, and so
  it converges to a continuous function $\phi$ on $I_{\eta_0}$.

    Since $(Z_n,{\bf 1}_{U_n})$ and $A_n$ are independent,
\begin{equation}\label{eq:kpt1}
\begin{split}
\lambda(s) \E[Z_n^{is }{\bf 1}_\S] &= \E[A_n^{is}]\E[Z_n^{is}\un]  - \lambda(s)\E[Z_n^{is} {\bf 1}_{U_n\setminus \S}]\\
&= \E[(A_nZ_n)^{is}\unn] + \E[(A_nZ_n)^{is}{\bf 1}_{U_n\setminus U_{n+1}}] - \lambda(s)\E[Z_n^{is} {\bf 1}_{U_n\setminus \S}].\end{split}
\end{equation}

Combining \eqref{eq:kpt1} with the well known inequality
\begin{equation}\label{eq:g4}
\big| e^{is} - e^{it}  \big| \le |s-t|, \qquad  s,t\in\R,
\end{equation}
 we obtain
  \begin{equation*}
  \begin{split}
    |\lambda(s)|^{n+1} |\phi_{n+1}(s) - \phi_n(s)| &=
    \big|\ES[Z_{n+1}^{is}] - \lambda(s)\ES[Z_n^{is}]\big|
=    \frac 1{\P[\S]}\big| \E[Z_{n+1}^{is}{\bf 1}_\S] - \lambda(s)\E[Z_n^{is}{\bf 1}_\S]\big|\\
&\le     C \big(
 \E[ \big| Z_{n+1}^{is} - (A_nZ_n)^{is}\big| \unn]  + \P[U_n\setminus \S]\big)\\
    &\le C\big(  s \E\big[|\log \Delta_n|\unn\big] + \P[U_n\setminus \S]\big).
  \end{split}
  \end{equation*}
  Lemma \ref{lem:sr1} and Lemma \ref{lem:lower ld2} imply that the above expression is bounded by $C \rho^n$ for some $\rho<1$.
Since  the function $\lambda$ is continuous and $\lambda(0)=1$, there exists a small neighbourhood of $0$ such that $\eqref{eq:s1}$ is satisfied with some
$\rho_0 < \rho$.

{\sc Step 2.} To prove \eqref{eq:o3} and differentiability of $\phi$, we proceed by induction. We will prove that for any  $j\le K$
there are $C>0$, $\eta_j>0$, $\rho_j<1$ such that
\begin{equation}\label{eq:pt6}
\big |\lambda(s)\big |^{n+1} \big| \phi ^{(j)}_{n+1}(s) - \phi ^{(j)}_n(s)  \big |\leq Cn^j\rho_j^n, \quad n\in \N,  s\in I_{\eta_{j}}.
\end{equation} If $j=1$ the above inequality implies  that the sequence of derivatives $\phi'_n$ converges uniformly to some function $\psi$
on $I_{\eta_j}$. Therefore $\phi'=\psi$ and  $\phi$ is continuously differentiable  (see e.g. Theorem 14.7.1 in \cite{Tao:Analysis}).
The same inductive  argument guarantees that $\phi\in C^{K}(I_{\eta_{K}})$ and since $\lambda$ is continuous with $\lambda(0)=1$, inequality
\eqref{eq:o3} follows directly from \eqref{eq:pt6} with $ \max_{j\le {K}}\rho_j\leq \rho <1$ and $\eta < \min_{j\le {K}} \eta_j$.

{\bf Proof of \eqref{eq:pt6} for $j=1$.}
%
Let us start with the following formula
\begin{equation}\label{eq:g3}
\lambda(s)^{n+1}\big( \phi_{n+1}(s) - \phi_n(s)
\big) = \ES[Z_{n+1}^{is}] - \lambda(s)\ES[Z_n^{is}].
\end{equation}
Denote by $L_n(s)$ and $R_n(s)$ the function on the left and the right side of the above equation. Differentiating  the function $L_n(s)$
we obtain
$$
L_n'(s) = (n+1)\lambda'(s)\lambda(s)^n (\phi_{n+1}(s) - \phi_n(s)) + \lambda(s)^{n+1}(\phi'_{n+1}(s) - \phi'_n(s))
$$
Since $|\lambda'(s)| \le \E |\log A| <\infty$, in view of \eqref{eq:s1},
\begin{align*}
|(n+1)\lambda '(s)\lambda (s)^n (\phi_{n+1}(s) - \phi_n(s))|&\leq C(n+1)\rho _0^n|\lambda (s)|^n\E |\log A|,\quad \mbox{for}\ s\in I_{\eta _0}\\
&\leq C\rho _1^n, \quad \mbox{for}\ s\in I_{\eta _1}
\end{align*}
where $\eta _1$ is taken small enough to guarantee $\rho _0|\lambda (s)|\leq \rho _1<1$ for $s\in I_{\eta _1}$.
Thus, it is sufficient to ensure that $R_n'(s)$ is exponentially small as $n$
tends to $+\infty$ uniformly for $s\in I_{\eta_1}$. We have
$$
R_n'(s) = \E_\S[i \log Z_{n+1}\cdot Z_{n+1}^{is}] - \lambda'(s)\E_\S[Z_n^{is}] - \lambda(s)\E_\S[ i\log Z_n \cdot Z_n^{is}].
$$
Moreover,
\begin{equation*}
\E_\S[i \log Z_{n+1}\cdot Z_{n+1}^{is}]=I_1-I_2,\end{equation*}
where
\begin{align*}
I_1=&\frac{1}{\P [\S]} \E[i \log Z_{n+1}\cdot Z_{n+1}^{is}{\bf 1}_{U_{n+1}}]\\
I_2=&\frac{1}{\P [\S]}\E[i \log Z_{n+1}\cdot Z_{n+1}^{is}{\bf 1}_{U_{n+1}\setminus \S}]
\end{align*}
and by Lemma \ref{lem:lower ld2}, $|I_2|\leq  C\rho^n$.
Similarly, recalling  independence of $(Z_n,{\bf 1}_{U_n})$ and $A_n$, we obtain
\begin{equation*}
\lambda'(s)\E_\S[Z_n^{is}] + \lambda(s)\E_\S[ i\log Z_n \cdot Z_n^{is}]= I_3+I_4,
\end{equation*}
where
\begin{align*}
I_3=&\frac{1}{\P [\S]}\left (  \E[i\log A_n\cdot (A_nZ_n)^{is}{\bf 1}_{U_{n}}] + \E[ i\log Z_n \cdot (A_nZ_n)^{is}{\bf 1}_{U_{n}}]\right )\\
I_4=& \lambda'(s)\frac{1}{\P [\S]}\E[Z_n^{is}{\bf 1}_{U_{n}\setminus \S}]+
\lambda(s)\frac{1}{\P [\S]}\E[i\log Z_n \cdot Z_n^{is}{\bf 1}_{U_{n}\setminus \S}]
\end{align*}
and $|I_4|\leq C\rho ^n$ by Lemma \ref{lem:lower ld2}. Moreover, $I_3=I_5+I_6$, where
\begin{equation*}
I_5= \frac{1}{\P [\S]}\left (  \E[i\log (A_nZ_n)\cdot (A_nZ_n)^{is}{\bf 1}_{U_{n+1}}]\right )
\end{equation*}
and
\begin{equation*}
|I_6|=\big|\E_\S[ i\log (A_n Z_n) \cdot (A_nZ_n)^{is}{\bf 1}_{U_n \setminus U_{n+1}}]\big)\big|  \le C \rho^n
\end{equation*}
again by Lemma \ref{lem:lower ld2}. Therefore, 
$$
| R_n'(s) - I_1+I_5| 
\le C \rho^n
$$ for some $\rho < 1$ and $C<\infty$.
Finally, using inequality \eqref{eq:g4}, we have
\begin{align*}
  \big| \E[i \log Z_{n+1}&\cdot Z_{n+1}^{is}{\bf 1}_{U_{n+1}}] - 
\E[ i\log (A_nZ_n) \cdot (A_nZ_n)^{is}{\bf 1}_{U_{n+1}}]\big|\\
  &\le \big| \E[ |\log Z_{n+1} - \log (A_n Z_n)|{\bf 1}_{U_{n+1}}] - 
\E[ |\log (A_nZ_n)| \cdot| Z_{n+1}^{is} - (A_nZ_n)^{is}{\bf 1}_{U_{n+1}}]\big|\\
& \le \big| \E[|\log \Delta_n| {\bf 1}_{U_{n+1}}]\big|
+ s \E[ |\log (A_nZ_n)| \cdot| \log \Delta_n |{\bf 1}_{U_{n+1}}]\big| \le C\rho^n,
\end{align*}
where the last inequality follows  from Lemma \ref{lem:sr1}.
Thus we obtain  \eqref{eq:pt6} for $j=1$ with $\rho _1=\max (\rho _0, \rho)$.



\medskip
{\bf Proof of \eqref{eq:pt6} for general $k\le K$.} Suppose that \eqref{eq:pt6} holds for $j\le k-1$
and recall \eqref{eq:g3}. Let
 $L_n^{(k)}(s), R_n^{(k)}(s)$, be the k th derivatives of the left and right hand side of \eqref{eq:g3}. Using the binomial formula for $L_n^{(k)}$,
to prove \eqref{eq:pt6} we need to ensure
\begin{equation}\label{eq:kp1}
\begin{split}
  \big| L_n^{(k)}(s) - &\lambda(s)^{n+1}\big( \phi^{(k)}_{n+1}(s) - \phi_n^{(k)}(s)\big) \big|\\
  &= \bigg|\sum_{j=0}^{k-1} {k\choose j} \big(\lambda(s)^{n+1}\big)^{(k-j)}\big( \phi^{(j)}_{n+1}(s) - \phi_n^{(j)}(s)
\big)\bigg| \le C n^k \rho^n
\end{split}
\end{equation}
and
\begin{equation}\label{eq:kp2}
\big|R_n^{(k)}(s)\big| \le C \rho^n
\end{equation}
for some $\rho <1$.

{\bf Proof of \eqref{eq:kp1}.}
 Note that
	\begin{equation}\label{eq:pt10}
	\bigg|\frac{(\lambda (s)^n)^{(p)}}{\lambda (s)^n}\bigg|\leq C_pn^{p},
	\end{equation} for sufficiently small $s$. Indeed, proceeding by  induction on $p$, we have
	\begin{align*}
	\bigg|\frac{(\lambda (s)^n)^{(p+1)}}{\lambda (s)^n}\bigg|=&\bigg| \frac{n(\lambda (s)^{n-1}\lambda '(s))^{(p)}}{\lambda (s)^n}\bigg|	
	=n\bigg| \sum _{j=0}^p
	{p\choose j}\frac{(\lambda (s)^{n-1})^{(j)}}{\lambda (s)^{n-1}}\frac{\lambda (s)^{(p-j+1)}}{\lambda (s)}\bigg|\\
	\leq &n\bigg(\sum _{j=0}^p
	{p\choose j}C_j(n-1)^j\bigg)\sup _{j,s}\bigg| \frac{\lambda (s)^{(p-j+1)}}{\lambda (s)}\bigg|\leq C_{p+1}n^{p+1},
	\end{align*}
which gives \eqref{eq:pt10}. Finally, by the induction hypothesis \eqref{eq:pt6} and \eqref{eq:pt10}
\begin{align*}
\left | \left (\lambda(s)^{n+1}\right )^{(k-j)}\right | \big| \phi ^{(j)}_{n+1}(s) - \phi ^{(j)}_n(s)  \big |=&
\left | \frac{\left (\lambda(s)^{n+1}\right )^{(k-j)}}{\lambda (s)^{n+1}}\right | \big |\lambda (s)^{n+1} \big |\big| \phi ^{(j)}_{n+1}(s) - \phi ^{(j)}_n(s)  \big |\\
&\leq C_{k-j}n^{k-j}Cn^j\rho ^n,
\end{align*} which gives \eqref{eq:kp1}.

{\bf Proof of \eqref{eq:kp2}.} Using ${\bf 1}_\S  = {\bf 1}_{U_n} - {\bf 1}_{U_n\setminus\S}$, we write $R_n^{(k)}$ as
 $$ R_n^{(k)}(s) = \frac 1{\P[\S]}\big(
  \E[Z_{n+1}^{is}{\bf 1}_{U_{n+1}}] - \E[(A_nZ_n)^{is}\un]
  -\E[Z_{n+1}^{is}{\bf 1}_{ U_{n+1} \setminus \S}]
   + \lambda(s) \E[Z_n^{is}{\bf 1}_{U_n\setminus \S}]
  \big)^{(k)}$$
  In view of Lemma \ref{lem:lower ld2} and (H1), the third and the fourth terms are exponentially small, i.e. they can by bounded by $C\rho^n$ uniformly for all $s\in \R$.
To estimate the difference of the remaining terms, invoking Lemma \ref{lem:lower ld2}, we write
\begin{align*}
\Big| \big( \E[Z_{n+1}^{is}&{\bf 1}_{U_{n+1}}] - \E[(A_nZ_n)^{is}\un] \big)^{(k)}\Big|\\
& \le   \Big| \E\big[ (i\log Z_{n+1})^k Z_{n+1}^{is}\unn\big] - \E\big[(i\log(A_nZ_n))^k(A_nZ_n)^{is}\unn \big] \Big|\\
&\qquad \qquad +\E\big[|\log(A_nZ_n))^k| {\bf 1}_{U_n\setminus U_{n+1}}] \\
  &\le  \E\big[ \big| (\log Z_{n+1})^k -(\log(A_nZ_n))^k\big| \unn\big]\\
  &\qquad\qquad +\E\big[ |\log (A_nZ_n)|^k \big| Z_{n+1}^{is} - (A_nZ_n)^{is}\big|\unn\big] +  C\rho^n.
  \end{align*}
  Now we apply the equality $a^k-b^k = (a-b)(a^{k-1} + a^{k-2}b+\ldots + b^{k-1})$ and \eqref{eq:g4}. Hence
  \begin{align*}
  \Big| \big( \E[Z_{n+1}^{is}{\bf 1}_{U_{n+1}}] - \E[(A_nZ_n)^{is}\un] \big)^{(k)}\Big|
  &\le  \sum_{j=0}^{k-1} \E\big[ (\log Z_{n+1})^j |\log(A_nZ_n)|^{k-1-j}|\log \Delta_n| \unn \big] \\
  &+
   s \E\Big[ |\log (A_nZ_n)|^k|\log \Delta_n|\unn  \Big] + C \rho^n.
\end{align*}
All the three terms can be bounded by $C \rho^n$ in view of 
  Lemma \ref{lem:sr1}, 
  which proves \eqref{eq:kp2}.
Hence the induction argument is complete and we conclude the Proposition.
\end{proof}

\section{Proofs of the limit theorems}
\begin{proof}[Proof of Theorem \ref{thm:clt}]
Since the second moment of $\log A$ is finite, the central limit theorem holds for the random walk $S_n = \log \Pi_n$.   In terms of characteristic functions we have
$$
\lambda^n\Big( \frac{s}{\sqrt n\sigma}\Big)e^{-is \sqrt n \mu/\sigma} = \E\Big[e^{is\; \frac{  S_n - n\mu }{\sqrt n \sigma}}\Big]
\to e^{-s^2/2} \quad \mbox{ as } n\to\infty
$$ for every $s\in \R$. Proposition \ref{prop:phi} entails that the sequence of functions $\phi_n(s) = \frac{\ES[Z_n^{is}]}{\lambda^n(s)}$
converges uniformly on some interval $I_{\eta}$ to a continuous function $\phi$ and $\phi(0)=1$. Therefore for every $s\in \R$
$$
\ES\Big[e^{is\; \frac{\log Z_n - n\mu }{\sqrt n \sigma}}\Big] =
 \phi_n\Big( \frac{s}{\sqrt n\sigma}\Big)\lambda^n\Big( \frac{s}{\sqrt n\sigma}\Big)e^{-is \sqrt n \mu/\sigma} \to e^{-s^2/2} \quad \mbox{ as } n\to \infty.
$$ Thus we conclude the result.
\end{proof}

The next lemma will be used in the proof of Theorem \ref{thm:edg}.
\begin{lem}\label{lem:nonl}
Assume that conditions (H1) and (H2) are satisfied, $\E (\log A)^2<\infty$ and
$\log A$ is nonlattice. Then, given $0<M<\infty$, there are $C_1, \chi >0$ and $\rho_1\in (0,1)$ such that
\begin{equation}\label{eq:q1}
  \sup_{ |s|\in[n^{-1\slash 3}, M]} \big|\ES Z^{is}_n \big|\le C_1e^{-\chi n^{1\slash 12}}.
\end{equation}
Moreover if \eqref{eq:pp1} holds, then for any $\delta, \gamma > 0$ there are $C_2>0$ and $\rho_2\in (0,1)$ such that
\begin{equation}\label{eq:q2}
  \sup_{|s|\in [\delta, n^\gamma]} \big|\ES Z^{is}_n \big| \le C_2\rho_2^n.
\end{equation}
\end{lem}
\begin{proof}
The argument is partially inspired by the proof of Lemma 2.5 in \cite{Grama:Liu}.
Let $\Pi_{j,n} = \prod_{k=j}^{n-1} A_k$ and $m=\lfloor n/4\rfloor$.
First we prove that there are constants $\rho<1$ and $C<\infty$
such that  
\begin{equation}\label{eq:sr33}
  \big| \ES (Z_{j+1}\Pi_{j+1,n})^{is} - \ES (Z_j\Pi_{j,n})^{is} \big| \le C (1 + |s|)\rho^{n\slash 4}
\end{equation}	
for $m\le j < n$,
and
\begin{equation}\label{eq:sr34}
  \big| \ES (Z_m\Pi_{m,n})^{is} \big| \le C ( \rho^{n\slash 4} + |\lambda(s)|^{3n/4}\big).
\end{equation}
Inequality \eqref{eq:sr33} follows directly from \eqref{eq:g4} and Lemmas  \ref{lem:sr1},
\ref{lem:lower ld2}:
\begin{align*}
\big| \ES (Z_{j+1}&\Pi_{j+1,n})^{is} - \ES (Z_j\Pi_{j,n})^{is} \big|\\
&\le
\frac 1{\P[\S]}\big| \E \big( e^{is \log (Z_{j+1}\Pi_{j+1,n})} - e^{is \log (Z_{j}\Pi_{j,n})} \big) {\bf 1}_{U_{j+1}} \big|
+  C \big| \P(U_{j+1}\setminus \S ) \big|\\
&\le
C |s| \E |\log \Delta_j| {\bf 1}_{U_{j+1}} + C \rho^j \le   C (1+|s|)\rho^j
\le   C (1+|s|)\rho^{n\slash 4}.
\end{align*}
To estimate \eqref{eq:sr34} we apply Lemma \ref{lem:lower ld2} and for $m=n\slash 4$ independence of $Z^{is}_m\cdot {\bf 1}_{U_m}$ and~$\Pi_{m,n}$ 
\begin{align*}
\big| \ES (Z_m\Pi_{m,n})^{is} \big|
&\le C \Big( \big| \E[ (Z_m\Pi_{m,n})^{is} {\bf 1}_{U_m} ]  \big|
+\big| \E[ (Z_m\Pi_{m,n})^{is} {\bf 1}_{U_m\setminus \S } ]  \big|
\Big)\\
&\le C\big( \big| \E [\Pi_{m,n}^{is}] \big| + \P[U_m\setminus \S ] \big)\\
&  \le C(    \big| \lambda(s)\big|^{3n/4} + \rho^{n\slash 4})
\end{align*}
which completes the proof of \eqref{eq:sr34}.

Since both inequalities \eqref{eq:sr33} and \eqref{eq:sr34} hold, we obtain
\begin{align*}
\big| \ES Z_n^{is} \big| &\le
\sum_{j=m}^{n-1}   \big| \ES (Z_{j+1}\Pi_{j+1,n})^{is} - \ES (Z_j\Pi_{j,n})^{is} \big| +
\big| \ES (Z_m\Pi_{m,n})^{is} \big|\\
&\le    C (1 + |s|)\rho^{n\slash 4} +  C |\lambda(s)|^{3n/4}.
\end{align*}
Now we expand $\lambda (s)$. Let $d=\E (\log A)^2> \mu ^2 $. We have
\begin{equation*}
\lambda (s)= 1+i\mu s -\frac{1}{2}ds^2 +iO(s^3)
\end{equation*}	
and so
\begin{equation*}
|\lambda (s)|^2= \left (1-\frac{1}{2}ds^2\right )^2 +(\mu s+O(s^3))^2=1-(d-\mu ^2)s^2+O(s^4).
\end{equation*}
Therefore, for some $\chi >0$,
\begin{equation*}
|\lambda (s)|\leq 1-2\chi s^2, \quad \mbox{when } \ |s|\leq \eps,
\end{equation*}
for $\eps $ sufficiently small. Hence for $n^{-1\slash 3}\le |s|\le \eps $,
\begin{equation*}
|\lambda (s)|^{3n\slash 4}\leq \left (1-2\chi n^{-2\slash 3}\right )^{3n\slash 4}\leq e^{-\chi n^{1\slash 12}}
\end{equation*}
for $n$ large enough.
 Since $\log A$ is nonlattice and the function $\lambda$ is continuous, we have
$\sup_{|s|\in[\delta, M]} |\lambda(s)| < 1$ which together with the last inequality entails \eqref{eq:q1}.  If we assume
additionally \eqref{eq:pp1}, then $\sup_{|s| \ge \delta} |\lambda(s)| < 1$ and thus we conclude \eqref{eq:q2}.

\end{proof}
\begin{proof}[Proof of Theorem \ref{thm:edg}]
We are going to adopt the proof 
of Theorem XVI.4.1 in \cite{feller:1971} based on the Berry-Esseen inequality (Lemma XVI.3.2 in \cite{feller:1971}).
For this kind of approach see also \cite{FMN}. Let $F$ be a distribution function with Fourier transform  $\widehat f$.
Suppose that a non negative function $g$ is integrable and bounded by $m$, 
with continuously differentiable Fourier transform $\widehat g$, $\widehat g(0)=1$,
$\widehat g'(0)=0$. 
Then by the Berry-Esseen inequality
\begin{equation}\label{eq:s2}
  |F(x) - G(x)| \le \frac 1{\pi} \int_{-T}^T \bigg|\frac{\widehat f(t) - \widehat g(t)}{t}\bigg| dt + \frac{24m}{\pi T},
\end{equation}
for all $x\in \R$  and for all $T>0$.
Denote by $F_n$ the distribution function of $\frac{\log Z_n - n\mu}{\sigma \sqrt n}$ conditioned on ${\mathcal S}$ and let $\widehat f_n$ be its characteristic function.

{\bf Step 1. Expansion of $\widehat f_n$.} First we will find an expansion of $\widehat f_n$. Writing $\widehat f_n$ in terms of the functions $\phi_n$ and $\lambda$ we have:
\begin{equation}\label{eq:ksr4}
\widehat f_n(t) = \ES\Big[ e^{it\frac{\log Z_n - n\mu}{\sigma \sqrt n}}  \Big]
=\ES\Big[ Z_n^{it/\sigma \sqrt n}\Big] e^{-\frac{it\mu\sqrt n}{\sigma}} = \phi_n\Big(
\frac{t}{\sqrt n\sigma} \Big) \lambda^n\Big( \frac{t}{\sqrt n\sigma }\Big)e^{-\frac{it\mu\sqrt n}{\sigma}}.
\end{equation}
 Recall that $\phi_n(s) = \E_{\mathcal S}[Z_n^{is}]/\lambda^n(s)$ and $\Lambda(s) = \log \lambda(s) = \log \E A^{is}$.  
By Proposition \ref{prop:phi}, in a small neighbourhood $I_\eta$ of 0, we  may expand $\phi_n$ into the Taylor series and approximating $\phi_n$ by $\phi$ we have
\begin{equation}\label{expphi}
 \phi_n(s)  = \sum_{k=0}^{r-1} \frac{\phi_n^{(k)}(0)s^k}{k!} + O(s^r)= 1+ \sum_{k=1}^{r-1} \frac{\phi^{(k)}(0)s^k}{k!} + O(s^r) + o(s \rho^n),\quad s\in I_{\eta}\end{equation}
for some $\rho<1$.
Since $\Lambda(0) = 0 $, $\Lambda'(0) = \lambda'(0) = i\mu $ and $\Lambda''(0) = \lambda''(0) - \lambda'(0)^2 = -\sigma^2$ we write also the Taylor expansion of $\Lambda$ near 0:
\begin{equation}\label{explam}
\Lambda(s) = \sum_{k=0}^r \frac{\Lambda^{(k)}(0)}{k!} s^k + o(s^r) = i\mu s - \frac{\sigma^2}{2}s^2 +
\sum_{k=3}^r \frac{\Lambda^{(k)}(0)}{k!} s^k + o(s^r).
\end{equation}
Combining both \eqref{expphi} and \eqref{explam} with the Taylor formula for $e^x$, for $|t|\leq \s n^{1\slash 6}$, we have 
\begin{align*}
\widehat   f_n(t) &=
  \phi_n\bigg( \frac t{\sigma \sqrt n} \bigg) e^{n\big( \Lambda(t/\sigma \sqrt n) -  it\mu / \sqrt n  \sigma   \big)}\\
  &=
 \bigg(
 1 + \sum_{k=1}^{r-1} \frac{\phi^{k}(0)}{k!} \bigg(\frac{t}{\sigma \sqrt n}\bigg)^k + O\bigg(\frac{t^r}{n^{r/2}}\bigg) + o\bigg( \frac {t\rho^n}{\sqrt n}  \bigg)
 \bigg)\\ &\qquad\qquad \times  e^{-\frac{t^2}2 +\sum_{k=3}^r \frac{\Lambda^{(k)}(0)}{\sigma^k k!} \frac{t^k}{n^{k/2-1}} + nO(t^{r+1}/ n^{(r+1)/2}) }\\
& =  e^{-\frac{t^2}2}\left ( 1 + \sum_{k=3}^r \frac{p_k(t)}{n^{k/2-1}} + (t^3+t^{r^2+r+1})o(n^{-r/2+1})\right), 
\end{align*} where $p_k = \sum_{j=1}^{s_k} a_{j,k}t^j$ are some polynomials of order $s_k\leq r(r-1)$,
whose coefficients $a_{j,k}$ depends only on derivatives of $\phi$ and $\Lambda$ at zero, e.g. $p_3(t) = \frac{t}{\sigma}\phi'(0)  + \frac{t^3}{6\sigma^3}\Lambda'''(0)$.

Indeed, let $W_r=\sum_{k=3}^r \frac{\Lambda^{(k)}(0)}{\sigma^k k!} \frac{t^k}{n^{k/2-1}}$. For $|t|\leq \s n^{1\slash 6}$, both $|W_r|$ and $nO(t^{r+1}/ n^{(r+1)/2})$ are bounded uniformly in $n$. Hence we may write
\begin{align*}
e^{W_r +
nO(t^{r+1}/ n^{(r+1)/2}) }=& \left (1+ \sum _{l=1}^{r-2}W_r^l/l!+ O(W_r^{r-1})\right )\left (1+nO(t^{r+1}/ n^{(r+1)/2})\right )\\
=&1+\sum_{k=3}^r \frac{\wt p_k(t)}{n^{k/2-1}} + (t^3+t^{r^2+1})o(n^{-r/2+1}),
\end{align*}
where the degree of $\wt p_k$ does not exceed $r(r-2)$. The latter combined with the above expansion of $\phi _n$ allows us to conclude.

{\bf Step 2. Construction of function $G_r$.}
Now we explain how to construct the polynomials $Q_k$ defining the function $G_r$ in \eqref{eq:pp2}. We want to find a function $G_r$ such that
\begin{equation}\label{eq:f8}
  \widehat g_r(t) = e^{-\frac{t^2}2}\bigg( 1 + \sum_{k=3}^r \frac{p_k(t)}{n^{k/2-1}}\bigg)
\end{equation} for $g_r = G_r'$.
For this purpose we recall the definition of Hermite polynomials:
\begin{equation}\label{eq:f6}
  H_j(x):= (-1)^j e^{\frac{x^2}2} \frac{\partial^j}{\partial x^j} \Big( e^{-x^2/2} \Big),\qquad j\in\N.
\end{equation}
An immediate calculation entails: $H_1(x)=x$, $H_2(x)=x^2-1$, $H_3(x)=x^3-3x$. Then
\begin{equation}\label{eq:f7}
  (\psi H_j)\widehat{\ } (t) = (it)^j e^{-\frac{t^2}{2}},
\end{equation}
where $\psi(x) = \frac1{\sqrt{2\pi}} e^{-x^2/2}$ is the density of a standard Gaussian variable.
Putting \newline $q_k(x) = \sum_{j=1}^{s_k} (-i)^j a_{j,k} H_j(x) $, where $a_{j,k}$ are coefficients of polynomials $p_k$,
we define
$$
g_r(x) = \psi(x) \bigg(      1+ \sum_{k=3}^r \frac{q_k(x)}{n^{k/2-1}}\bigg).
$$ Since the Fourier transform is linear, invoking \eqref{eq:f7}, we conclude equation \eqref{eq:f8}.
 $g_r$ is integrable, $\widehat g_r$ is continuously differentiable, moreover  $\widehat g_r(0)=1$ and
$\widehat g_r'(0)=0$, therefore we define $G_r(x) = \int_{-\infty}^x g_r(y)dy$.
 We still need to define $Q_k$. 
They are primitive functions to $q_k$ and they can be expressed
in terms of Hermite functions
\begin{equation}\label{qpoly}
Q_k(x) = -\sum_{j=1}^{s_k} (-i)^j  a_{j,k} H_{j-1}(x),
\end{equation}
thanks to the formula
$$\big( H_{j-1}(x) \psi(x)  \big)' = - H_j(x) \psi(x), \qquad \mbox{ for } j\ge 1,
$$ which is a consequence of
\eqref{eq:f6}. For example, one can compute for $k=3$, $Q_3(x) = \frac{i\phi'(0)}{\sigma} + \frac{i\Lambda'''(0)}{6\sigma^3} H_2(x)$.

{\bf Step 3. Conclusion. } Fix an $\varepsilon >0$.  Appealing to \eqref{eq:s2} with $T = a n^{r/2-1}$ and $a = 24m_r/\pi\varepsilon$ (here $m_r = \sup_{x\in\R}
g_r(x)$)
 \begin{align*}
\big| F_n(x) - G_r(x) \big| &\le \frac 1{\pi} \int_{\{ |t| <  \sigma  n^{1\slash 6} \}} \bigg|\frac{\widehat f_n(t)-\widehat g_r(t)}{t}\bigg|dt\\
&+ \frac 1{\pi} \int_{\{  \sigma n^{1\slash 6} \le |t| < T\}} \bigg|\frac{\widehat f_n(t)-\widehat g_r(t)}{t}\bigg|dt + \frac{24m_r}{\pi T}\\ &:= I_1+I_2+I_3.
\end{align*}
The third term $I_3$ is equal to  $\varepsilon n^{-r/2+1}$. To estimate  $I_1$  notice that
$$
|\widehat f_n(t) - \widehat g_r(t)| =
e^{-\frac{t^2}2} (t^2+t^{r^2+r}) o(n^{-r/2+1}).
$$ Hence $I_1 = o(n^{-r/2+1})$.
Finally,  in view Lemma \ref{lem:nonl} and \eqref{eq:f8}, we bound $I_2$ by
 \begin{align*}
I_2 &\le  \frac{1}{ \sigma  \pi \sqrt n} \int_{ \{n^{-1\slash 3} \le |s| < T/\sigma \sqrt n\}} \big|\ES Z_n^{is}\big| ds + C e^{- \sigma ^2 n^{1\slash 3}\slash 2}\\
 &\le C T e^{-\chi n^{1\slash 12}}  + C e^{- \sigma ^2 n^{1\slash 3}\slash 2} = C (1/\varepsilon + 1) o(n^{-r/2+1}).
 \end{align*}
 Notice that for $r\geq 4$, $T/\sqrt{n}\to \8 $ and so we need the condition \eqref{eq:pp1} and the second part of Lemma \ref{lem:nonl}.
 Summarizing, we have just proved that, uniformly with respect to $x\in \R$, we have
 $$
\limsup_{n\to\infty } n^{r/2-1}\big| F_n(x) - G_r(x) \big| \le \varepsilon.
 $$
Passing with  $\varepsilon$ to 0, we   complete proof of the Theorem \ref{thm:edg}.
\end{proof}

\begin{proof}[Proof of Theorem \ref{thm:ren}] Thanks to Proposition \ref{prop:phi} the proof essentially follows from the arguments presented in \cite{Breiman} (Chapter 10).
  For any finite interval $I$, we denote by
  $$ N(I) = \#\{ n:\; \log Z_n \in I \}
  $$ the amount of time the logarithm of the population spends in $I$. Let $U$ be the corresponding renewal measure, i.e.
  $$
  U(I) = \ES N(I) = \sum_{n=1}^{\infty} \P_{\mathcal S}(\log Z_n \in I).
  $$ Our aim is to prove that the renewal measure converges vaguely  to a scaled Lebesgue measure:
  $$
  \lim_{y\to\infty} U(I+y) = \frac{|I|}{\mu}.
   $$
 Let $\H$ denotes the class of complex valued and integrable functions on $\R$ whose Fourier  transforms
   are real and compactly supported. Note that if $h\in \H$, then the Fourier inversion formula is valid
   $$
   h(x) = \int_\R e^{ivx} \widehat{h}(v)dv.
   $$
    Let $h_y = h(\cdot - y)$.
   Due to Theorem 10.7 in \cite{Breiman} it is sufficient
    to prove that
   \begin{equation}\label{Leb}
      \lim_{y\to\infty} U(h_y) = \frac 1{\mu}{\rm Leb}(h),
      \end{equation}
   for $h\in \H$. 

First for $r\in(0,1)$ we define the finite measures
   $$
   U_r(B) = \sum_{n=1}^\infty r^n \P_\S[\log Z_n \in B], \qquad  B\in {\mathcal B}(\R^+).
   $$
For any $h\in \H$, by the Fourier inversion formula, we have
   \begin{equation}\label{eq:kor1}
   U_r(h) = \sum_{n=0}^{\infty} r^n \ES\big[ h(\log Z_n) \big] = \sum_{n=0}^\infty r^n \ES \bigg[
   \int_\R Z_n^{is}\widehat h(s) ds
   \bigg]
   = \sum_{n=0}^\infty r^n    \int_\R \widehat h(s) \ES Z_n^{is} ds.
   \end{equation}
   Recall that in view of Proposition \ref{prop:phi}, for $s$ in a small neighbourhood $I_{\eta}$ of 0 we have
   \begin{equation}\label{eq:kor2}
   \big| \ES[Z_n^{is}] - \lambda^n(s) \phi(s) \big| \le C \rho^n
   \end{equation}
    for some $\rho<1$.    Fix a function $h\in \H$ and denote by $I$ the support of $\widehat h$.
   Since
   $\widehat h_y(s) = e^{-isy}\widehat h(s)$, we may write \eqref{eq:kor1} as
    $$
    U_r(h_y) = I_1 + I_2 + I_3,
    $$
    where
    \begin{align*}
    I_1 &= \int_{I_{\eta}} e^{-iys} \cdot \frac{\widehat h(s)\phi(s)}{1-r\lambda(s)} ds, \\
    I_2 &= \int_{ I_{\eta}} e^{-isy} \widehat h(s) \sum_{n=0}^{\infty} r^n \big( \ES Z_n^{is} - \lambda(s)^n \phi(s) \big) ds,\\
    I_3 & = \int_{I \cap I_{\eta}^c} e^{-isy} \widehat h(s) \sum_{n=0}^\infty r^n \ES[Z_n^{is}]ds.
    \end{align*}
    The first part $I_1$ determines \eqref{Leb}, whereas two other terms are negligible. Indeed, repeating literally arguments present in \cite{Breiman} (pages 221-224)
    one proves
    $$
    \lim_{y\to\infty} \lim_{r\to 1} I_1= \frac{2\pi}{\mu}\widehat h(0).
    $$ We skip the details. To deal with $I_2$ denote
    $$
    g(s) = {\bf 1}_{I_{\eta}} \widehat h(s) \sum_{n=0}^{\infty}\big( \ES Z_n^{is} - \lambda(s)^n \phi(s) \big).
    $$ Then, by \eqref{eq:kor2} and the Lebesgue dominated convergence, we have
    $$
    \lim_{r\to 1} I_2 = \int_\R e^{-isy} g(s)ds = \widehat g(y).
    $$ Hence, the Riemann-Lebesgue Lemma entails
    $$\lim_{y\to\infty}\lim_{r\to 1} I_2 = \lim_{y\to\infty} \widehat h(y) = 0.
    $$
    Similarly, appealing to Lemma \ref{lem:nonl}, we estimate $I_3$.
Thus,     the proof is complete.
\end{proof}

\appendix
\section{Proof of Lemmas \ref{lem:lower} and \ref{zet}}

We start with two auxiliary results.

\begin{lem}\label{lem:bin}
		For every $\xi >0$ there are $\sigma \in (0,1/2)$ and $C>0$ such that
		\begin{equation*}
		\sum _{k\leq \s n}\binom{n}{k}\leq C(1+\xi) ^n, \qquad n\in\N.
		\end{equation*}
	\end{lem}
\begin{proof}
Given $\sigma \in (0,1/2)$ whose precise value will be specified below, define  $m=\lfloor \s n\rfloor$.
Since the sequence $k\to {n\choose k}$ is increasing for $k\le n/2$ we have
\begin{equation*}
\sum _{k\leq \s n}\binom{n}{k}\leq m \binom{n}{m }.
\end{equation*}
Recalling the Stirling formula $n!\sim \sqrt{2\pi n} (n/e)^n$ we derive the asymptotic behavior, as $n\to \infty$, of the binomial term:
\begin{align*}
\binom{n}{m}&\sim \frac 1{\sqrt{2\pi}} \frac{\sqrt{n}}{\sqrt{m}\sqrt{n-m}}\left (\frac{n}{e}\right )^n \left (\frac{m}{e}\right )^{-m}\left (\frac{n-m}{e}\right )^{-(n-m)}\\
&\leq\frac C  {\sqrt{\sigma n}}   n^n (\s n-1)^{-\s n +1} (n-\s n -1)^{-(n-\s n-1)}\\
&\leq C n^{3\slash 2} \sigma^{-1/2} \big(  \s^{-\s }	(1-\s)^{-(1-\s) }\big)^n
\end{align*}
Since the function $x\to x^x$ converges to 1 both as $x\to 0^+$ and $x\to 1^-$, there is $\s<1/2$ such that $\s^{-\s }	(1-\s)^{-(1-\s)}  < 1+\xi$. Thus the proof is completed.
\end{proof}

To state the next  result denote $\P_k[\cdot ] = \P[\cdot |Z_0=k]$. Define
the set
$${\mathcal I} = \big\{ j\ge 1:\; \P(Q(j)>0, Q(0)>0)>0
\big\}$$
to be the set of integers $j$ such that with positive probability $0$ and $j$ are  in the support of the measure $Q$ i.e. with positive probability the initial particle
may have both $0$ and $j$ descendants.
Next, let us introduce the set of integers which can be reached from $\mathcal I$:
$$ {\rm Cl}({\mathcal I}) = \big\{ k\ge 1:\; \P_j[Z_n=k]>0 \mbox{ for some } j\in {\mathcal I}
 \mbox{ and } n\in\N_0 \big\}.$$
The following large deviations lemma was proved in \cite{BansBoningsmall} (Theorem 2.1 and Proposition 2.2).
This is the only result of \cite{BansBoningsmall} that we quote as it is:

\begin{lem}\label{lem:lower bansaye}
Assume $\P[Z_1=0]>0$. Then for any $k,j \in {\rm Cl}({\mathcal I})$ the following limit exists
$$
\rho = -\lim_{n\to\infty} \frac 1n \log \P_k[Z_n = j]
$$ and the limit does not depend on  the choice of $k$ and $j$.

Moreover if $\P_{\mathcal Q}[Z_1=0]<\gamma$, $P$ a.s. for some $\gamma < 1$ and $\E [ |\log A|]<\infty$, then $\rho >0$.
\end{lem}

\begin{proof}[Proof of Lemma \ref{lem:lower}]
	The proof is split into 4 steps, the crucial argument is contained in Step 4. Basically we need to consider the path $Z_1,...,Z_n$ and estimate for some $n_0$
	\begin{equation}\label{largepath}
	\P [Z_n\leq e^{\theta n}, Z_1,...,Z_n\geq n_0]\end{equation}
	and
	\begin{equation*}
		\P [Z_n\leq e^{\theta n}, Z_1,...,Z_i< n_0]\quad \mbox{for}\quad i>n\slash 2.
	\end{equation*}
	Steps 1-3 contain some preparatory calculations for \eqref{largepath}. If $\P (Z_1\geq 1)=1$, $n_0$ may be taken 1 and the proof considerably simplifies.
	
    {\sc Step 1.} Fix $0<\sigma<\min\{ 1-\gamma,1/2\}$. First we prove that for every $\eps >0$  there are  $\beta_1 > 0$ and $C_1>0$ such that
\begin{equation}\label{eq:qq2}
\E_n\bigg[ \bigg(\frac{Z_1}n\bigg)^{-\varepsilon} {\bf 1}_{\{Z_1 \ge 1 \}} {\bf 1}_{B_n}\big| {\mathcal Q}\bigg] \le C_1  e^{-\beta_1 n},\quad P  \mbox{ a.s.}
\end{equation}
for
	$$
	B_n=\{ \mbox{at most}\ \lfloor \s n\rfloor \ \mbox{among}\ \xi_{1,1},\ldots \xi_{n,1} \ \mbox{are not equal to}\ 0  \}.$$
Let us emphasize that $Z_1$ in \eqref{eq:qq2} denotes the population at the first generation of a process initiated
with $n$ individuals.

Choose $\xi \le 1/\sqrt {\gamma} - 1$.
Then	by Lemma \ref{lem:bin}
$$	\P_{\mathcal Q} (B_n)\leq \sum _{k\leq \s n}\binom{n}{k}  \P_{\mathcal Q}[\xi_{1,1}>0]^k \P_{\mathcal Q}[\xi_{1,1}=0]^{n-k}
	\leq \gamma^{(1-\s)n} \sum _{k\leq \s n}\binom{n}{k} \le C \delta^n,$$
with $\delta = 	(1+\xi) \gamma ^{(1-\s)}<1$.
	Hence
	$$
\E_n \bigg[\bigg(\frac{Z_1}{n}\bigg)^{{-\varepsilon}}\Ind {\{ Z_1\geq 1\}}\Ind {B_n}\big|{\mathcal Q}\bigg]
\leq Cn^{\varepsilon}\delta^n$$
and we conclude \eqref{eq:qq2} with $C_1 = C \sup_n n^{\varepsilon} \delta^{n/2}$, $e^{-\beta_1} = \delta^{1/2}$.

{\sc Step 2.} We will prove that there are $\eps _0,\beta_2>0$ and $n_0 \in\N$  such that
\begin{equation}\label{eq:q3}
  \E_n\bigg[ \bigg( \frac{Z_1}{n} \bigg)^{-\varepsilon} {\bf 1}_{\{Z_1 \ge 1\}}   \bigg]
 \leq e^{-\beta_2}
	\end{equation}
for all $n\ge n_0$ and $0<\eps \le \eps _0$.


Suppose that $\P (Z_1=0)>0$. Then $\{\sum _{i=1}^n \xi_{i,n}\leq \s n\}\subset B_n$ and so by \eqref{eq:qq2}
\begin{equation}\label{eq:q7}
\E_n \bigg[ \bigg(\frac{Z_1}{n} \bigg)^{{-\varepsilon}}\Ind {\{1\leq Z_1 \leq \s n\}}\bigg]
\le C e^{-\beta_1 n}
\end{equation} for all $n\ge 1$. If $\P (Z_1=0)=0$ then the left hand side of \eqref{eq:q7} is 0. 

By the reverse Fatou lemma
\begin{equation}\label{eq:q8}
\begin{split}
\limsup_{n\to\infty} \E_n&\bigg[ \bigg( \frac{Z_1}{n} \bigg)^{-\varepsilon} {\bf 1}_{\{ Z_1 >\sigma n \}} \bigg]\\
&\le \E\bigg[ \E_{\mathcal Q}\bigg[ \limsup_{n\to\infty} \bigg(\frac 1n \sum_{i=1}^n {\xi_{i,1}}\bigg)^{-\varepsilon}
 {\bf 1}_{\{ \sum_{i=1}^n \xi_{i,1}>\sigma n \}}   \bigg]   \bigg] = \E A^{-\varepsilon}.
\end{split}
\end{equation}
Indeed, the last equality follows  from the strong law of large numbers, because given ${\mathcal Q}$
$$
\lim_{n\to\infty}  \bigg(\frac 1n \sum_{i=1}^n {\xi_{i,1}}\bigg)^{-\varepsilon}
 {\bf 1}_{\{ \sum_{i=1}^n \xi_{i,1}>\sigma n \}}  = A^{-\varepsilon},\qquad \P_{\mathcal Q} \mbox{ a.s.,}
$$ since $\sigma$ has been chosen  such that $\sigma < 1-\gamma$ and $1-\gamma \leq A$, $P$ a.s.
Moreover, $\E A^{-\eps}<1$ provided $\eps $ is small enough (see the comments below \eqref{eq:a}). Therefore
combining \eqref{eq:q7} with \eqref{eq:q8} we obtain \eqref{eq:q3} for $n$ greater equal then some $n_0$.

If $\P (Z_1=0)=0$, and so $\P (Z_1=1)<1$, $n_0$ may be taken 1. Indeed, the left hand side
of \eqref{eq:q3} is always strictly smaller then 1.

{\sc Step 3.} Let $\beta _2, \eps, n_0$ be as in Step 2. We prove by induction that 
\begin{equation}\label{eq:q24}
  \E_j\big[ Z_n^{-\varepsilon} {\bf 1}_{\{Z_1\ge n_0,\ldots, Z_n \ge n_0\}}  \big] \le n_0  e^{-\beta_2 (n-1)},
\end{equation} for any $n\ge 1$ and $j\le n_0$.

Indeed, for $n=1$ and $j\le n_0$  we write
	\begin{align*}
	\E_j \big[ Z_1^{{-\varepsilon}} \Ind {\{ Z_1\geq n_0\}}\big]= & \E\big[ \left (\xi_{1,1}+...+\xi_{j,1}\right )^{{-\varepsilon}} \Ind {\{ \xi_{1,1}+...+\xi_{j,1}\geq n_0\}}\big]\\
 \leq  & \E\bigg[ \left (\xi_{1,1}+...+\xi_{j,1}\right )^{{-\varepsilon}} \cdot \sum_{i=1}^j {\bf 1}_{ \xi_{i,1}\ge 1\}}\big]\\
	&\leq \sum _{i=1}^j\E \big[ \xi_{i,1}^{{-\varepsilon}}\Ind {\{\xi_{i,1}\geq 1 \}}\big]	=j\E \big[\xi_{1,1}^{{-\varepsilon}}\Ind {\{\xi_{1,1}\geq 1 \}}\big] \leq n_0,
\end{align*}
where for the first inequality above we use the simple observation that if
 $\sum_{i=1}^j\xi_{i,1}\ge n_0$, then at least one of random variables $\xi_{i,1}$ must be greater or equal to 1.

For arbitrary $n$, appealing to \eqref{eq:q3} and the induction hypothesis, we derive
\begin{align*}
\E_j\big[ &Z_n^{{-\varepsilon}} \Ind {\{ Z_1\ge n_0,..., Z_n\geq n_0\}}\big] \\& =
\E_j \bigg[ Z_{n-1}^{{-\varepsilon}}\E_{Z_{n-1}}
\bigg[  \bigg( Z_{n-1}^{-1}\sum_{i=1}^{Z_{n-1}}\xi_{i,n}\bigg)^{{-\varepsilon}}\Ind {\{ Z_n\geq n_0\}} \bigg]
\Ind {\{ Z_1\geq n_0,..., Z_{n-1}\geq n_0\}}\bigg] \\
& \overset{\eqref{eq:q3}}{\le}  \E_j \big[ Z_{n-1}^{{-\varepsilon}} e^{-\beta_2} \Ind {\{ Z_1\geq n_0,..., Z_{n-1}\geq n_0\}}\big]   \overset{\rm ind. hyp.}{\le}  n_0 e^{-\beta_2 (n-1)}.
\end{align*} 
This proves \eqref{eq:q24}.

  {\sc Step 4. Proof of \eqref{eq:q10}.} Now we adapt to our settings arguments contained in the proof of Lemma 4.3, \cite{BansBoninglower}.
  We start with a slightly weaker version of \eqref{eq:q10}. Let $n_0$ be as in Step 2. Then  there are
  $\theta_1 <\E \log A$, and $\beta_3 > 0$ such that
  \begin{equation}\label{eq:q26}
    \P_j\big[ Z_n \le e^{\theta_1 n}, Z_1 \ge n_0, \ldots, Z_n \ge n_0
    \big] \le Cn_0e^{-\beta _3n},
  \end{equation} for any $j\leq n _0<n $.
  Indeed, in view of \eqref{eq:q24} we have
    \begin{multline*}
     \P_j\big[ Z_n \le e^{\theta_1 n}, Z_1 \ge n_0, \ldots, Z_n \ge n_0
    \big] =\P_j \big[Z^{{-\varepsilon}}_n\geq e^{-\varepsilon\theta_1 n}, Z_1\ge n_0,...,Z_n\geq n_0\big]\\
\leq e^{\varepsilon\theta_1 n} \cdot  \E_j\big[  Z_n^{{-\varepsilon}}\Ind {\{ Z_1\geq n_0,...,Z_n\geq n_0\}}\big]
 \overset{\eqref{eq:q24}}{\leq}    n_0 e^{\eps \theta _1} e^{(\eps \theta _1-\beta_2) (n-1)}.
\end{multline*} Choosing $\theta_1 < \min \{\beta_2/\varepsilon, \E\log A\}$ we conclude \eqref{eq:q26} and so \eqref{eq:q10} follows in the case $\P (Z_1=0)=0$.

Now fix $k\in {\rm Cl}({\mathcal I})$. We will prove
\begin{equation}\label{eq:q12}
  \P_k\big[ 1\le Z_n \le e^{\theta n}  \big] \le C e^{-\beta_3 n}.
\end{equation}
  Note that \eqref{eq:q12} entails the Lemma.
 Indeed, given an environment ${\mathcal Q}$, let $Z_n^1, \ldots, Z_n^k$ be independent copies of $Z_n$. Then, by the Jensen inequality,
 for some $\theta < \theta_1$ and large $n$, we have
\begin{align*}
  \P[1\le Z_n \le e^{\theta n}]^k &\le \E\Big[ \P_{{\mathcal Q}}[1\le Z_n \le e^{\theta n}]^k \Big]\\
 &=  \E\Big[ \P_{{\mathcal Q}}[1\le Z_n^i \le e^{\theta n} \mbox{ for }i=1,\ldots, k] \Big]\\
 & \le \P_k[k\le Z_{n+1} \le k e^{\theta n}]
  \le \P_k[1\le Z_{n+1} \le  e^{\theta_1 (n+1)}] \overset{\eqref{eq:q12}}{\le} Cn_0 e^{-\beta_3 (n+1)}
\end{align*}
which gives \eqref{eq:q10}.

  To prove \eqref{eq:q12} we again use the induction argument.
  Let $$\tau=\inf \{ i\leq n: Z_{i+1}\geq n_0, ..., Z_n\geq n_0\},$$ that is $ Z_{\tau}<n_0, Z_{\tau+1}\geq n_0,..., Z_n\geq n_0$ or, in other
  words, $\tau$ is the largest index smaller than $n_0$ such that $ Z_{\tau}<n_0$, but all the other elements with larger indices are
  greater or equal to $n_0$.
	If for every $i\leq n$, $Z_n<n_0$ then $\tau =n$. Therefore,
$$
\P_k\big[1\le Z_n \le e^{\theta n}  \big]
\le \sum_{i=1}^{n-1}\P_k\big[ 1\le Z_n \le e^{\theta n}, \tau = i \big] + \P_k\big[ 1\le Z_n \le n_0 \big].
$$	
  In view of Lemma \ref{lem:lower bansaye} the second term is smaller than $C e^{-\beta_4 n}$ for some $\beta_4>0$.
  To estimate the first term we divide all the summands into two sets. If  $i>n\slash 2$, an appeal to Lemma \ref{lem:lower bansaye}
  entails
  \begin{equation}\label{eq:q14}
    \P_k\big[ 1\le Z_n \le e^{\theta n}, \tau = i \big] \le \P_k[1\leq Z_{i-1} < n_0] \le C e^{-\beta_4 (i-1)} \le C e^{-\beta_5 n}.
  \end{equation}
  Finally for $i\le n\slash 2$, invoking \eqref{eq:q26}, we obtain
  \begin{equation}\label{eq:q15}
  \begin{split}
    \P_k\big[ 1&\le Z_n \le e^{\theta n}, \tau = i \big] \\
&    \le
    \P_k \big[ 1\le Z_{i-1} < n_0 \big]\sup_{j\le n_0}
    \P_j\big[ 1\le Z_{n-i} \le e^{\theta n}, Z_1\ge n_0,\ldots, Z_{n-i}\ge n_0 \big]\\
    & \overset{\eqref{eq:q26}}{\le}  C n_0 e^{-\beta_3 n\slash 2}.
    \end{split}\end{equation}
 Inequality  \eqref{eq:q15} together with \eqref{eq:q14} imply \eqref{eq:q12} and we conclude the Lemma.
\end{proof}

\begin{proof}[Proof of Lemma \ref{zet}]
 We start with 
 the following observation:
 in view of \eqref{eq:wt1}, there is  $C>0$ such that for any $s\in[0,q)$, $k\in [1,q)$ and $r\geq 1$, we have
\begin{equation}\label{eq:cz1}
  \sup_n \E\big[ |\log (A_n Z_n )|^s  (\log \Delta_n)^r {\bf 1}_{\{ \Delta_n\ge 1 \}} {\bf 1}_{U_{n+1}}  \big] <C.
\end{equation}

The proof is split into two steps.

{\sc Step 1.} First we prove the Lemma for $k= 1$. 
Since on the set $\{\Delta_n < 1\}$ we have $Z_{n+1} < A_n Z_n$, we can estimate
\begin{align*}
  \E\big[ & \log Z_{n+1}   {\bf 1}_{U_{n+1}} \big]  =
  \E\big[ \log Z_{n+1} {\bf 1}_{\{\Delta_n < 1\}}{\bf 1}_{U_{n+1}} \big]
   + \E\big[ \log Z_{n+1} {\bf 1}_{\{\Delta_n \ge  1\}}  {\bf 1}_{U_{n+1}} \big]\\
    &\le  \E\big[ \log (A_n Z_n) {\bf 1}_{\{\Delta_n < 1\}}{\bf 1}_{U_{n+1}} \big]
    + \E\big[ \big ( \log\Delta_n +  |\log (A_n Z_n)|\big) {\bf 1}_{\{\Delta_n \ge  1\}}  {\bf 1}_{U_{n+1}} \big]\\
&\le  \E\big[ |\log (A_n Z_n)| {\bf 1}_{U_{n}} \big]
    + \E\big[ \big ( \log\Delta_n \big) {\bf 1}_{\{\Delta_n \ge  1\}}  {\bf 1}_{U_{n+1}} \big]\\
&\le  \E\big[ (\log  Z_n) {\bf 1}_{U_{n}} \big]
+ \E\big[ |\log A_n | {\bf 1}_{U_{n}} \big]
    + C,
\end{align*}
where the last inequality is a consequence of \eqref{eq:cz1}. In view of hypothesis (H1) we conclude
$$  \E\big[ \log Z_{n+1} {\bf 1}_{U_{n+1}} \big]
\le  \E\big[ \log  Z_n {\bf 1}_{U_{n}} \big] + C_1
$$ for some $C_1>0$ and all $n\in\N$ which, by an easy induction argument, yields
$$
\E\big[ (\log  Z_n) {\bf 1}_{U_{n}} \big] \le C_1 n.
$$

\medskip

{\sc Step 2.} Now proceeding by induction we prove the Lemma for an arbitrary $k\le q$. Choose $k>1$ and assume that \eqref{logmoment} holds for $k-1$.
We will proceed similarly as above, however this time we will rely on the inequality
\begin{equation}\label{eq:in}
(x+y)^k\leq x^k+2^k(x^{k-1}y+y^k), \quad \mbox{for}\quad x,y\geq 0, k>1.
\end{equation}
Indeed, either $x\leq y$ and then $(x+y)^k\leq 2^ky^k$, or $0\leq y<x$, we write $x+y=\frac{x-y}{x}x+\frac{y}{x}2x$ and then by the convexity of $x^k$,
$((x+y)^k-x^k)/y\leq ((2x)^k-x^k)/x =(2^k-1)x^{k-1}$, that entails \eqref{eq:in}.

 Now we use again that
 $\{Z_{n+1} < A_n Z_n\}  = \{\Delta_n < 1\}$ and applying twice \eqref{eq:in}, \eqref{eq:cz1}, we have
 \begin{align*}
    \E\big[  (\log & Z_{n+1})^k  {\bf 1}_{U_{n+1}} \big]  =
  \E\big[ (\log Z_{n+1})^k {\bf 1}_{\{\Delta_n < 1\}}{\bf 1}_{U_{n+1}} \big]
    +\E\big[ (\log Z_{n+1})^k {\bf 1}_{\{\Delta_n \ge  1\}}  {\bf 1}_{U_{n+1}} \big]\\
    &\le  \E\big[ (\log (A_n Z_n))^k {\bf 1}_{\{\Delta_n < 1\}}{\bf 1}_{U_{n+1}} \big]
    + \E\big[ \big ( |\log (A_n Z_n)|+ \log\Delta_n \big)^k {\bf 1}_{\{\Delta_n \ge  1\}}  {\bf 1}_{U_{n+1}} \big]\\
    &\overset{\eqref{eq:in}}{\le}
        \E\big[ |\log (A_n Z_{n})| ^k  {\bf 1}_{U_{n}} \big]
        + 2^k \E \big[  |\log(A_n Z_n)|^{k-1} \log \Delta_n  {\bf 1}_{\{\Delta_n \ge 1\}} {\bf 1}_{U_{n+1}}   \big]\\
       & + 2^k \E \big[  \big( \log \Delta_n \big)^k  {\bf 1}_{\{\Delta_n \ge 1\}} {\bf 1}_{U_{n+1}}   \big]\\
  &\overset{\eqref{eq:cz1}}{\le}
\E\big[ |\log (A_n Z_{n})| ^k  {\bf 1}_{U_{n}} \big]      + C.
 \end{align*}
Next referring again to \eqref{eq:in}, we have 
 \begin{align*}
    \E\big[  (\log Z_{n+1})^k & {\bf 1}_{U_{n+1}} \big]  \le
\E\big[ \big( \log Z_n + |\log A_n| \big)^k  {\bf 1}_{U_{n}} \big]      + C\\
 &\overset{\eqref{eq:in}}{\le} \E\big[  (\log Z_n)^k  {\bf 1}_{U_{n}} \big]  +  2^k \E\big[ ( \log Z_n)^{k-1} |\log A_n| {\bf 1}_{U_{n}} \big]
 + 2^k \E\big[ |\log A_n|^k \big]      + C.
 \end{align*}
 Now, since $(Z_n, {\bf 1}_{U_n})$ and $A_n$ are independent, in view of (H1) and the induction hypothesis, we write
  \begin{align*}
  \E\big[  (\log Z_{n+1})^k  {\bf 1}_{U_{n+1}} \big] & \le
 \E\big[  (\log Z_n)^k  {\bf 1}_{U_{n}} \big]  +  2^k \E\big[ ( \log Z_n)^{k-1} {\bf 1}_{U_{n}} \big]\E[|\log A_n|]\\ &+ 2^k \E\big[ |\log A_n|^k \big] + C\\
 &\le
 \E\big[  (\log Z_n)^k  {\bf 1}_{U_{n}} \big]  +  C n^{k-1}.
 \end{align*}
which easily entails the Lemma.

\end{proof}

\vspace{5mm}

\noindent {\bf Acknowledgement}.
We thank two anonymous referees for careful reading and  numerous helpful suggestions.
 We are also grateful  to Piotr Dyszewski for useful comments and discussions. Ewa Damek incorporated some ideas communicated her by Konrad Kolesko during the work on \cite{damek:natert:kolesko}.
The  research  was  partially supported by the National Science Center, Poland (grant number  2019/33/ B/ST1/00207)

\end{document}